\documentclass[final]{siamltex}
\usepackage{amsmath,amsfonts,amssymb}
\usepackage{latexsym}
\usepackage{epsfig,overpic}
\usepackage{showkeys}
\usepackage{todonotes}
\usepackage{booktabs}
\usepackage{hyperref}

\usepackage{pgfplots}
\usetikzlibrary{shapes,arrows,snakes,calendar,matrix,backgrounds,folding,calc,positioning,spy}
\usepackage{tikz}

\usepackage[scientific-notation=true]{siunitx}
\newcommand{\N}{\mathcal N}

\newcommand{\T}{\mathcal T}
\newcommand{\F}{\mathcal F}
\newcommand{\E}{\mathcal E}
\newcommand{\id}{{\rm id}}

\newcommand{\V}{\mathcal V}
\newcommand{\Vreg}{V_{\text{reg},h}}

\newcommand{\Vregphi}{V_{\text{reg}}}
\newcommand{\Vregphih}{V_{\text{reg},h}}

\newcommand{\rr}{\mathbb{R}}

\newcommand{\jumpleft}{[\![}
\newcommand{\jumpright}{]\!]}
\newcommand{\jump}[1]{\jumpleft #1 \jumpright}
\newcommand{\spacejump}[1]{\jump{#1}}
\newcommand{\averageleft}{\{\!\!\{}
\newcommand{\averageright}{\}\!\!\}}
\newcommand{\average}[1]{\averageleft #1 \averageright}

\newcommand{\hphi}{\hat \phi_h}
\newcommand{\lin}{{\text{\tiny lin}}}
\newcommand{\Omegalin}{\Omega^{\lin}}
\newcommand{\hatOmegalin}{\hat\Omega^{\lin}}
\newcommand{\Gammalin}{\Gamma^{\lin}}
\newcommand{\OGamma}{\Omega^{\Gamma}}

\newcommand{\TGamma}{{\mathcal T}^{\Gamma}}

\newcommand{\thetah}{\Theta_h}
\newcommand{\thetahGamma}{\Theta_h^\Gamma}

\DeclareGraphicsExtensions{.pdf,.eps,.ps,.eps.gz,.ps.gz,.eps.Y}

\newtheorem{remark}{Remark}
\title{$L^2$-error analysis of an isoparametric unfitted finite element method for elliptic interface problems}
\author{Christoph Lehrenfeld\thanks{Institut f\"ur Numerische und Angewandte Mathematik, University of G\"ottingen, D-37083 G\"ottingen,
Germany; email: {\tt lehrenfeld@math.uni-goettingen.de}}
\and Arnold Reusken\thanks{Institut f\"ur
Geometrie und Praktische  Mathematik, RWTH Aachen University, D-52056 Aachen,
Germany; email: {\tt reusken@igpm.rwth-aachen.de}}
}

\begin{document}
\maketitle
\begin{abstract}
In the context of unfitted finite element discretizations the realization of high order methods is challenging due to the fact that the geometry approximation has to be sufficiently accurate. 
Recently a new  unfitted finite element method was introduced which achieves a high order approximation of the geometry for domains which are  implicitly described by smooth level set functions. This method is based on  a parametric mapping which transforms a piecewise planar interface (or surface) reconstruction to a high order approximation.  In the paper [C. Lehrenfeld, A. Reusken, \emph{Analysis of a High Order Finite Element Method for Elliptic Interface Problems}, arXiv 1602.02970, Accepted for publication in IMA J. Numer. Anal.] an  a priori error analysis of the method applied to an interface problem is presented. The analysis reveals optimal order discretization error bounds in the $H^1$-norm. In this paper we extend this analysis and derive optimal $L^2$-error bounds.
\end{abstract}


%\tableofcontents

\section[Introduction]{Introduction}\label{sec:introduction}
In  recent years there has been a strong increase in the research on  development and analysis of \emph{unfitted finite element methods}, cf. for example
\cite{bastian2009unfitted,
  % Becker20093352,
  % burman2012fictitious,
  burman2015cut,
  deckelnick14,
  % fries2010extended,
  % grande2014highorder,
  gross04,
  grossreusken07,
  hansbo2002unfitted,
  kummer13,
  % massjung12,
  olshanskii2009finite}. 
The design, realization and analysis of unfitted finite elements methods which are \emph{higher order} accurate is a challenging task. Especially in a setting where geometries are only implicitly described, e.g. by a zero level of a level set function, accurate numerical integration is difficult.
Different techniques have been developed to overcome this difficulty \cite{parvizianduesterrank07,muller2013highly,sudhakar2013quadrature,saye2015hoquad,cheng2010higher,dreau2010studied,fries2015,boiveau2016fictitious,burman2015cut,grande2014highorder}.
For a further discussion of the relevant literature  we refer to  \cite[Section 1.2]{CLARH1}.

Recently, in \cite{lehrenfeld15}, we introduced a new high order unfitted finite element method for scalar interface problems based on isoparametric mappings which is suitable for level set descriptions and higher order unfitted finite elements. The method is geometry based and is applicable to a wide range of problems, e.g. Stokes interface problems \cite{LPWL_PAMM_2016}, fictitious domain problems \cite{L_ARXIV_2016} or surface PDEs \cite{GLR_ARXIV_2016}. 

A detailed a-priori error analysis of the method applied to a scalar elliptic interface model problem is given in \cite{CLARH1}. Optimal $H^1$-norm discretization error bounds are derived for the case of a polygonal domain. In this paper we consider the same elliptic interface problem as in \cite{CLARH1} and extend the existing analysis in two directions. First, we allow for a \emph{piecewise smooth}, not necessarily polygonal, \emph{domain},  which is approximated with a higher order isoparametric boundary approximation. Secondly, we complement the existing analysis with \emph{optimal $L^2$-estimates}. 

% The two main contributions of this paper are the extension of main results from \cite{CLARH1} to allow for a high order   isoparametric boundary approximation and the derivation of an optimal order $L^2$-error bound.

%The method is geometry-based and can be applied to unfitted interface or boundary value problems as well as to partial differential equations on surfaces.  
%The key idea of the method is to construct a parametric mapping of the underlying triangulation based on a finite element approximation of the level set function which characterizes the interface (or surface). This mapping, which is easy to construct, uses the information available in a \emph{higher order} finite element approximation of the level set function  to map a piecewise planar interface approximation of the interface (or surface) to a \emph{higher order} approximation. This mapping defines an isoparametric unfitted  finite element space which is then used for the discretization of the partial differential equation. A precise explanation of this mapping and the discretization method is given in the Sections~\ref{sectmeshtransform} and \ref{unfittedFEM}. In \cite{lehrenfeld15} results of numerical experiments are presented which clearly demonstrate the higher order accuracy of the new method, but no rigorous error analysis of the method is given.

The paper is organized as follows. In Section \ref{prelim} we introduce the model interface problem. We assume a level set description of the interface and discuss  a standard  piecewise planar interface approximation in Section~\ref{sec:lset}. A key ingredient in the higher order unfitted finite element method is the isoparametric mapping used to obtain higher order geometry approximations. In Section~\ref{isoparammap} we present the construction of this mapping and the isoparametric unfitted finite element method for the elliptic interface problem. In particular it is explained how the same extension technique, which is a standard component in isoparametric finite element methods for high order approximation of smooth boundaries, is applied both for the boundary and interface mesh transformation. 
A main contribution of this paper is the error analysis of the method, i.e. the derivation of optimal $H^1$- and $L^2$-error bounds. The $H^1$-error results obtained  in \cite{CLARH1} (for a polygonal domain) form the basis for this analysis.  In Section~\ref{sectOld} we recall the most important results from \cite{CLARH1} and adapt these to the case of a higher order isoparametric boundary approximation.
The $L^2$ error analysis is presented in Section~\ref{sec:erroranalysis}. Finally, in Section~\ref{sec:numex} we give results of a numerical experiment. 

We are aware of the fact that parts of the paper contain results that are essentially the same as in \cite{CLARH1}. In particular the material given in the sections~\ref{sec:lset}, \ref{sectlocal}, \ref{sectextension} and \ref{sectmethod} can also be found in \cite{CLARH1}. We decided to include this material to make this paper more self-contained and thus improve its readability.

\section{Model interface problem}\label{prelim} We consider the model problem
\begin{subequations} \label{eq:ellmodel}
\begin{align}
- \mathrm{div} (\alpha_i \nabla {u}) &= \, f_i 
\quad \text{in}~~ \Omega_i , ~i=1,2, \label{eq:ellmodel1} \\
\spacejump{{\alpha} \nabla {u} }_{\Gamma} \cdot n_\Gamma &= \, 0, \quad \spacejump{{u}}_{\Gamma} = 0 \quad \text{on}~~\Gamma, \label{eq:ellmodel2} \\
u &= 0 \quad \text{ on } \partial \Omega.
 \end{align}
\end{subequations}
Here, $\Omega_1 \cup \Omega_2= \Omega \subset \Bbb{R}^d$, $d=2,3$, is a nonoverlapping partitioning of the domain, $\Gamma = \overline{\Omega}_1 \cap \overline{\Omega}_2$ is the interface and  $\spacejump{\cdot}_{\Gamma}$ denotes the usual jump operator across  $\Gamma$. The source terms $f_i$ are given on each subdomain $\Omega_i$, $i=1,2$. In the remainder we also use the source term $f$ on $\Omega$ which we define as $f|_{\Omega_i}=f_i$, $i=1,2$. We assume $\Gamma$ to be characterized as the zero level of a level set function (not necessarily a distance function).
The diffusion coefficient $\alpha$ is assumed to be piecewise constant, i.e. it has a constant value $\alpha_i > 0$  on each sub-domain $\Omega_i$. The weak formulation of this problem is as follows: determine $u \in  H_0^1(\Omega)$ such that
\[
 \int_{\Omega} \alpha \nabla u \cdot \nabla v \, dx = \int_\Omega f v \, dx \quad \text{for all}~~v \in H_0^1(\Omega).
\]
%We define the isoparametric Nitsche unfitted FEM  as a transformed version of the original Nitsche unfitted FE discretization \cite{hansbo2002unfitted} with respect to the interface approximation $\Gamma_h = \Theta_h(\Gammalin)$. Here $\Theta_h$ is a special transformation, which maps an easy to construct piecewise linear approximation $\Gammalin$ of $\Gamma$ to a higher order interface approximation.
%We explain the construction of this mapping in subsection~\ref{sectdefiso} below.
In the error analysis we assume that $\overline{\Omega}_1$ is strictly contained in $\Omega$, i.e., $\Gamma= \partial \Omega_1$ and ${\rm dist}(\Gamma,\partial \Omega)>0$. Furthermore, we assume $\Gamma$ to be (globally) smooth and $\partial \Omega$ to be piecewise smooth. 

\section{Boundary parametrization and level set representation of the interface}\label{sec:lset}
The domain $\Omega$, which has a (piecewise) smooth boundary is approximated with a family of \emph{polygonal} approximations $\{\Omega_h^{\lin}\}_{h>0}$ corresponding to 
 a family $\{\T_h\}_{h>0}$ of simplicial shape regular triangulations of $\{\Omega_h^{\lin}\}_{h>0}$ which are not fitted to $\Gamma$, cf.~Fig.~\ref{fig:trafos}. In the analysis we assume quasi-uniformity of the triangulations, hence,  $h \sim h_T:={\rm diam}(T),~T \in \T_h$. In the remainder we take a fixed $\Omega_h^{\lin}$ with corresponding triangulation $\T_h$ and, to simplify the presentation, drop the index $h$.

The standard finite element space of continuous piecewise polynomials up to degree $k$ with respect to the domain $\Omegalin$ is denoted by $V_h^k$:
\[
  V_h^k:= \{ \, v_h \in C(\Omegalin)~|~{v_h}_{|T} \in \mathcal{P}_k \quad \text{for all}~ T\in \T, \quad {v_h}_{|\partial \Omegalin}=0\, \}.
\]
 The nodal interpolation operator in $V_h^k$ is denoted by $I_k$. In the remainder we take a fixed $k \geq 1$. We are particularly interested in  $k \geq 2$. This $k$ denotes the polynomial degree of the finite element functions used in our finite element method (explained in Section~\ref{sectmethod}). To obtain a high order finite element discretization of \eqref{eq:ellmodel} we need sufficiently accurate representations of $\partial \Omega$ and $\Gamma$. We now addresss these representations.

For simplicity  we assume that all boundary  vertices of $\partial \T = \partial \Omega^{\lin}$ are located on $\partial \Omega$ (this assumption is not essential, cf. \cite{BrennerScott}, Sect. 4.7). 
 On $\partial \T$ we assume a given  parametrization of the exact boundary $\partial \Omega$ of the form 
\begin{equation} \label{boundparam} g_b(x)=x+\chi_b(x),
\end{equation}
 such that $g_b: \, \partial \T \to \partial \Omega$ is a bijection. We assume that  $\partial \Omega$ has smoothness properties such that the function $\chi_b$ can be chosen  piecewise smooth on $\partial \T$ and satisfies 
\begin{equation} \label{eq:assumptionchi}
  \|\chi_b\|_{\infty, \partial \T} + h \|D\chi_b\|_{\infty,\partial \T} \lesssim h^2 ~\text{and}~ \max_{F \in \F(\partial \T)} \|D^l \chi_b\|_{\infty,F} \lesssim 1,~l=2,..,k+1.
\end{equation}
holds. Here $\F(\partial \T)$ denotes the set of all  boundary edges ($d=2$) or boundary faces ($d=3$). Note that $\chi_b(x_i)=0$ at all boundary vertices $x_i$ of $\partial \Omega^{\lin}$. The function $\chi_b$ will be input for the parametric mapping used in our method.

We assume that the smooth interface $\Gamma$ is the zero level of a smooth level set function $\phi:\, \Omega \to  \Bbb{R}$, i.e.,  $\Gamma= \{\, x \in \Omega~|~\phi(x)=0\,\}$ and $ \Omega_i = \{x \in \Omega~| ~\phi(x) \lessgtr 0 \}$. 
This level set function is not necessarily close to a distance function, but has the usual properties of a level set function: 
\begin{equation} \label{LSdef}
\|\nabla \phi(x)\| \sim 1~,~~\|D^2 \phi(x)\| \leq c \quad \text{for all}~x~~ \text{in a neighborhood $U$ of $\Gamma$}.
\end{equation}
In the error analysis we assume that the level set function has the smoothness property  $\phi \in C^{k+2}(U)$. 
As input for the parametric mapping we need an approximation $\phi_h \in V_h^k$ of $\phi$, and we assume that this approximation satisfies the error estimate
\begin{equation} \label{err2}
  \max_{T\in \T} |\phi_h - \phi|_{m,\infty,T \cap U} \lesssim h^{k+1-m},\quad 0 \leq m \leq k+1.
\end{equation}
Here $|\cdot|_{m,\infty,T\cap U}$ denotes the usual semi-norm on the Sobolev space $H^{m,\infty}(T\cap U)$ and the constant used in $\lesssim$ depends on $\phi$ but is independent of $h$.
Note that  \eqref{err2} implies the estimate
\begin{equation} \label{err1}
 \|\phi_h - \phi\|_{\infty, U} + h\|\nabla(\phi_h - \phi)\|_{\infty, U} \lesssim h^{k+1}. 
\end{equation}
The zero level of the finite element function $\phi_h$ (implicitly) characterizes the discrete interface. 
\emph{The piecewise linear nodal interpolation of $\phi_h$ is denoted by $\hphi = I_1 \phi_h$.} Hence, $\hphi(x_i)=\phi_h(x_i)$ at all vertices $x_i$ in the triangulation $\T$. 
The low order geometry approximation of the interface, which is needed in our discretization  method, is the zero level of this function:
$$\Gammalin := \{ \hphi = 0\}.$$ 
The subdomains corresponding to $\Gammalin$ are denoted by $ \Omega^{\lin}_{i}= \{x \in \Omegalin\,|\, \hphi(x) \lessgtr 0 \}$. Note that $\Omega^{\lin}_{1} \cup \Omega^{\lin}_{2} =\Omegalin \neq \Omega$. All elements in the triangulation $\T$ which are cut by $\Gammalin$ are collected in the set $\TGamma := \{T \in \T \, |\, T \cap \Gammalin \neq \emptyset \}$. The corresponding domain is $\OGamma := \{ x \in T \, |\, T\in \TGamma\}$.
% The extended set which includes all direct neighbors to elements in $\TGamma$ is $\TGammaplus := \{T \in \T \, |\, \mathrm{meas}_{d-1}(T\cap \OGamma)\neq 0 \}$ with the corresponding domain $\OGammaplus := \{ x \in T\, |\, T \in \TGammaplus \}$.
% All elements in the triangulation that touch the boundary $\partial \Omegalin$ form the set $\Tbound := \{ T \in \T \,|\, T \cap \partial \Omegalin \neq \emptyset \}$.  The corresponding domain is $\Obound := \{ x \in T \,|\, T \in \Tbound\}$.
%Finally, we  introduce $\Oplus = \OGammaplus \cup \Obound$.

We note that $\partial \Omegalin$ and   $\Gammalin$ are (only)  second order accurate approximations of $\partial \Omega$ and $\Gamma$, respectively. 
To achieve higher order accuracy with respect to the interface and the boundary approximation we consider a parametric mesh transformation in section \ref{sectdefiso}.
\section{Isoparametric mapping}\label{isoparammap}
The isoparametric mapping $\thetah$, defined on $\Omegalin$,  that is used in the finite element method is based on a  \emph{local} mapping $\thetahGamma$ on $\OGamma$ which is combined with a suitable \emph{extension technique}. The latter is taken such that $\thetahGamma$ is ``smoothly'' extended and $\thetah$ yields a sufficiently accurate boundary approximation. In the error analysis we need a mapping $\Psi:\,\Omegalin \to \Omega$ which is sufficiently close (in a higher order sense) to $\thetah$ and has the properties $\Psi(\Gammalin)=\Gamma$, $\Psi(\partial \Omegalin)=\partial \Omega$. The construction of this $\Psi$ is also based on a \emph{local} mapping $\Psi^\Gamma$, that is very similar to $\thetahGamma$, which is  extended with the same technique as used in the extension of $\thetahGamma$. The mappings $\thetah$ and $\Psi$ for the case \emph{without boundary approximation} (i.e., $\Omega$ is polygonal) are explained in detail in \cite{CLARH1}. In the subsections below we 
 explain how the 
construction of \cite{CLARH1} can be extended to allow for isoparametric boundary approximation.

In section~\ref{sectlocal} we recall the definitions of the local mappings $\thetahGamma$, $\Psi^\Gamma$, which are the same as in  \cite{CLARH1}. In section~\ref{sectextension}  we discuss a general boundary data extension technique, which is essentially the same as the extension method used in isoparametric finite elements \cite{lenoir1986optimal,bernardi1989optimal}. In section~\ref{sectdefiso} this extension technique is applied to both $\thetahGamma$ and $\Psi^\Gamma$ resulting in the global mappings $\thetah$ and $\Psi$. We give some results, e.g., on the smoothness of the extensions and on the approximation error in  $ \thetah \approx \Psi$, which are derived in \cite{CLARH1}. These results are used in the error analysis.

\subsection{Local mappings} \label{sectlocal}
In the construction of the local mapping $\Theta_h^\Gamma$ we need a projection step from a function which is piecewise polynomial but possibly discontinuous (across element interfaces) to the space of continuous finite element functions. 
Let $C(\TGamma) := \bigoplus\limits_{T\in \TGamma} C(T)$ and $V_h^k(\OGamma) := V_h^k|_{\OGamma}$. We introduce a projection operator $P_h^\Gamma: C(\TGamma)^d \rightarrow V_h^k(\OGamma)^d$. 
The projection operator relies on a nodal representation of the finite element space $V_h^k(\OGamma)$.
%However, similar operators can also be defined for different bases.
The set of finite element nodes $x_i$ in $\TGamma$ is denoted by $\N(\TGamma)$, and $\N(T)$ denotes the set of finite element nodes associated to $T \in \TGamma $. All elements $T \in \TGamma$ which contain the same finite element node $x_i$ form the set denoted by %$\omega(x_i)$: 
%\begin{align*}
$ \omega(x_i)  := \{\, T \in \TGamma ~|~ x_i \in \N(T) \,\}, ~~ x_i \in \N(\TGamma)$.
Let $|\cdot|$ denote the cardinality of the set $\omega(x_i)$ and let $\psi_i$ be the nodal basis function corresponding to $x_i$. We define the projection operator $P_h^\Gamma: C(\TGamma)^d \to V_h^k(\OGamma)^d$ as 
\begin{equation} 
P_h^\Gamma v:= \!\!\!\!\! \sum_{x_i \in \N(\TGamma)} \!\!\!\!\! A_{x_i}(v) \, \psi_i, \text{ with }
  A_{x_i}(v):= \frac{1}{|\omega(x_i)|} \sum_{T \in \omega(x_i)} v_{|T}(x_i), ~~x_i \in \N(\TGamma).
\end{equation}
   %
% % \end{align*}
% For each finite element node we define the local average as 
% \begin{equation*} % \label{average}
%   A_{x_i}(v):= \frac{1}{|\omega(x_i)|} \sum_{T \in \omega(x_i)} v_{|T}(x_i), ~~x_i \in \N(\TGamma).
% \end{equation*}
% where $|\cdot|$ denotes the cardinality of the set $\omega(x_i)$.
% The projection operator $P_h^\Gamma: C(\TGamma)^d \to V_h^k(\OGamma)^d$ is defined as
% \begin{equation*} 
% P_h^\Gamma v:= \sum_{x_i \in \N(\TGamma)} A_{x_i}(v) \, \psi_i, \quad v \in C(\TGamma)^d,
% \end{equation*}
   %    where $\psi_i$ is the nodal basis function corresponding to $x_i$.
This is a simple and well-known 
projection operator considered also in e.g., \cite[Eqs.(25)-(26)]{oswald} and \cite{ernguermond15}.
The  construction of $\Theta_h^\Gamma$ is motivated by the following mapping $\Psi^\Gamma \in C(\OGamma)$, defined in \eqref{psi1} below. We introduce  the search direction $G:=\nabla \phi$ 
and a  function $d: \OGamma \to \rr$ defined as follows:  $d(x)$ is the (in absolute value) smallest number such that
\begin{equation} \label{cond1}
  \phi(x + d(x) G(x))=\hat \phi_h(x)  \quad \text{for}~~x \in \OGamma.
\end{equation}
(Recall that $\hat\phi_h$ is the piecewise linear nodal interpolation of $\phi_h$.)
Given the function $dG \in C(\OGamma)^d \cap H^{1,\infty}(\OGamma)^d$ we define:
\begin{equation} \label{psi1}
 \Psi^\Gamma(x):= x + d(x) G(x), \quad x \in \OGamma.
\end{equation}
Note that the function $d$ and mapping $\Psi^\Gamma$ depend on $h$, through $\hat \phi_h$ in \eqref{cond1}. We do not show this dependence in our notation. 
By construction the mapping $ \Psi^\Gamma$ has the property $ \Psi^\Gamma(\Gammalin)=\Gamma$. Further properties of $\Psi^\Gamma$ are derived in \cite[Corollary 3.2 and Lemma 3.4]{CLARH1} and given in the following lemma.
\begin{lemma} \label{lem:boundpsig}
  The following holds:
  \begin{subequations}
  \begin{align} 
  |d(x_i)| & \lesssim h^{k+1} \quad \text{for all {\rm vertices} $x_i$ of}~ T \in \TGamma ,\label{resd4a0}\\
 \|\Psi^\Gamma - \id \|_{\infty,\OGamma} +h \|D\Psi^\Gamma - I\|_{\infty,\OGamma} & \lesssim h^2, \label{fistderPsi}\\
   \max_{T \in \TGamma} \Vert D^l \Psi^\Gamma \Vert_{\infty,T} & \lesssim 1, \quad l \leq  k+1.\label{fistderPsiA}
    % \Vert D^l \PsiGammainv \Vert_{\infty,\PsiGamma(T)} \lesssim 1, \quad l = 1, .., k+1.
  \end{align}
\end{subequations}
% \todo[inline]{Do we really need the bounds for the inverse trafo?}
\end{lemma}
The mapping $\Psi^\Gamma$ will be used in the error analysis, cf. section~\ref{sectdefiso}.
 
We now explain the construction of $\Theta_h^\Gamma$, which consists of two steps.  In the first step we introduce a discrete analogon of $\Psi^\Gamma$ defined in \eqref{psi1}, denoted by $\Psi_h^\Gamma$. 
% On each element $T$ we will have that $\Vert \Psi_h \Vert_{m,\infty,T} \lesssim 1$, $m \leq k$, $T \in \TGamma$.
Based on this $\Psi_h^\Gamma$, which can be discontinuous across element interfaces,  we obtain a continuous transformation $\thetahGamma \in C(\OGamma)^d$ by averaging with the projection operator $P_h^\Gamma$. For the construction of $\Psi_h^\Gamma$ we need an efficiently  computable and accurate approximation of $G=\nabla \phi$ on $\TGamma$.
For this we consider the following two options
  \begin{equation} \label{eq:gh}
    G_h(x) = \nabla \phi_h(x), \quad \text{or}~~G_h(x) = (P_h^\Gamma \nabla \phi_h)(x), \quad x \in T \in \TGamma.
  \end{equation}
 Let $\mathcal{E}_T \phi_h$ be the polynomial extension of $\phi_h|_T$. 
We define  a function $d_h: \TGamma \to [-\delta,\delta]
$, with $\delta > 0$ sufficiently small, as follows: $d_h(x)$ is the (in absolute value) smallest number such that  
  \begin{equation} \label{eq:psihmap}
    \mathcal{E}_T \phi_h(x + d_h(x) G_h(x)) = \hat \phi_h(x), \quad \text{for}~~ x\in  T \in \TGamma.
  \end{equation}
Clearly, this $d_h(x)$ is a ``reasonable'' approximation of the steplength $d(x)$ defined in \eqref{cond1}. Given the function $d_h G_h \in  C(\TGamma)^d$ we define
 \begin{equation} \label{eq:psih}
    \Psi_h^\Gamma(x) := x + d_h(x) G_h(x) \quad \text{for}~x \in T \in \TGamma,
  \end{equation}
which approximates the function $\Psi^\Gamma$ defined in \eqref{psi1}.
To remove possible discontinuities of $\Psi_h^\Gamma$ in $\OGamma$ we apply the projection to obtain
\begin{equation}
  \thetahGamma:= P_h^\Gamma \Psi_h^\Gamma = \id + P_h^\Gamma(d_h G_h).
\end{equation}
The approximation result in the following lemma is from Lemma 3.6 in \cite{CLARH1}.
\begin{lemma}\label{lem4}
The estimate
\begin{equation}
  \sum_{r=0}^{k+1} h^r \max_{T \in \TGamma}\Vert D^r( \thetahGamma - \Psi^\Gamma) \Vert_{\infty,T} 
\lesssim h^{k+1}
\end{equation}
holds.
\end{lemma}

\subsection{Finite element boundary data extension method} \label{sectextension}
Let $\hat\T \subset \T$ be a subset of the triangulation such that $ \hatOmegalin :=\cup_{T \in \hat\T} T$ is a connected subdomain of $\Omegalin$. We also need the notation $\partial \hat \T= \partial \hatOmegalin$, $\hat \T^{\rm bnd}:=\{\, T\in \hat \T~|~ T \cap \partial \hat \T \neq \emptyset\}$, $\hat \T^{\rm int}= \hat \T \setminus\hat \T^{\rm bnd}$, cf. Figure \ref{fig:figt1t2} (left). Given such a connected triangulation $\hat \T$, boundary data $g \in C(\partial \hatOmegalin)$ and a polynomial degree $k \geq 1$, a \emph{linear extension operator} $\E_k(\hat  \T,g)$ is treated in \cite{CLARH1}, which is essentially the same as the finite element extension technique studied in \cite{lenoir1986optimal,bernardi1989optimal}. This type of extension operator is a standard component in isoparametric finite element methods for high order boundary approximations.

\begin{figure}
  \begin{center}
    \hfill
    \begin{tikzpicture}[scale=0.95]
      \input{figuresubdom}
      \draw[] (2.25,1.5) -- (2.25,1) -- (2.75,1.25) -- cycle ;
      \node[right,scale=1] at (2.75,1.25) {$T\!\in\!\hat \T$};
      \draw[fill=blue!35!white] (2.25,0.5) -- (2.25,0) -- (2.75,0.25) -- cycle ;
      \node[right,scale=1] at (2.75,0.25) {$T\!\in\!\hat \T^{\rm int}$};
      \draw[fill=yellow!50] (2.25,-0.5) -- (2.25,-1) -- (2.75,-0.75) -- cycle ;
      \node[right,scale=1] at (2.75,-0.75) {$T\!\in\!\hat \T^{\rm bnd}$};
      % \draw[fill=blue,fill opacity=0.35] (3,-0.5) -- (3,-1) -- (3.5,-0.75) -- cycle ;
      % \node[right,scale=1] at (3.5,-0.75) {$\T_2$};
      \draw[ultra thick, red] (2.25,-1.75)  -- (2.75,-1.75) node[right,scale=1]{\color{black}$\partial \hat \T$};
    \end{tikzpicture}
    \hfill
    \begin{tikzpicture}[scale=0.95]
      \input{figuret1t2}
      \draw[fill=red,fill opacity=0.45] (2.25,1.5) -- (2.25,1) -- (2.75,1.25) -- cycle ;
      \node[right,scale=1] at (2.75,1.25) {$T\!\in\!\T^\Gamma$};
      \draw[fill=green,fill opacity=0.25] (2.25,0.5) -- (2.25,0) -- (2.75,0.25) -- cycle ;
      \node[right,scale=1] at (2.75,0.25) {$T\!\in\!\T_1$};
      \draw[fill=blue,fill opacity=0.35] (2.25,-0.5) -- (2.25,-1) -- (2.75,-0.75) -- cycle ;
      \node[right,scale=1] at (2.75,-0.75) {$T\!\in\!\T_2$};
      \draw[thick] (2.25,-1.75)  -- (2.75,-1.75) node[right,scale=1]{$\Gammalin$};
    \end{tikzpicture}
    \hfill
  \end{center}
  \caption{Connected subtriangulation $\hat \T \subset \T$ with $\hat \T^{\rm int}$, $\hat \T^{\rm bnd}$ and $\partial \hat \T$ (left) and decomposition of a triangulation $\T$ into $\T^\Gamma$, $\T_1$ and $\T_2$ (right).}
  \label{fig:figt1t2}
\end{figure}

This operator is  local, meaning that it is applied elementwise for $T\in \hat T^{\rm bnd}$. Its construction is hierarchical in the sense that first an extension based on values at vertices lying on $\partial \hat \T$ is determined, then an extension from edges  lying on $\partial \hat \T$ and finally (for the three-dimensional case) an extension of the data from faces lying on $\partial \hat \T$ is determined. For a precise definition of this operator we refer to \cite{CLARH1}. In the following lemma we summarize some properties of this operator which are derived in  \cite{CLARH1,lenoir1986optimal,bernardi1989optimal}.
\begin{lemma}
  Let $\V(\partial \hat \T)$ denote the set of vertices in $\partial\hat \T$ and $\F(\partial\hat \T)$ the set of all edges ($d=2$) or faces ($d=3$) in $\partial\hat \T$. For $\E(g):=\E_k(\hat \T,g)$ the following holds
  \begin{subequations}
\begin{align}
  \E(g) &\in C(\hatOmegalin), \quad \E(g)=0~~\text{on $\hat T^{\rm int}$}, ~~
  \E(g)_{|\partial \hat \T}  =g. \label{prop1} \\
   \text{For all $T \in \hat \T^{\rm bnd}$~:}~&~\text{~if  $T \cap \partial \hat \T \in \V(\partial \hat \T)$ then $~\E(g) \in \mathcal{P}_1(T)$}. \label{prop2}\\
  \text{For all $T \in \hat \T^{\rm bnd}$~:}~ &~  \text{if $~g \in \mathcal{P}_k(T \cap \partial \hat \T)~$ then  $~\E(g) \in \mathcal{P}_k(T)$}. \label{prop3}
\end{align}
Furthermore, 
\begin{align}
  \|D^n\mathcal{E}(g)\|_{\infty,\hatOmegalin}   & \lesssim \max_{F \in \F(\partial \hat \T)} \sum_{r=n}^{k+1} h^{r-n} \|D^rg\|_{\infty,F} \nonumber \\ & \quad +\, h^{-n}\max_{x_i \in \V(\partial \hat \T)}|g(x_i)| ,\quad n=0,1, \label{estext1} \\
  \max_{T \in \hat\T} \|D^n\mathcal{E}(g)\|_{\infty,T} & \lesssim \sum_{r=n}^{k+1} h^{r-n} \|D^r g\|_{\infty,F}, \quad n=2,\ldots, k+1,\label{estext2}
\end{align}
\end{subequations}
for all $g \in C(\partial \hat T)$ such that $g \in C^{k+1}(F)$ for all $F \in \F(\partial \hat \T)$. 
\end{lemma}
\ \\[1ex]
The results in \eqref{prop1}-\eqref{prop3} give basic structural properties of the extension operator. In particular the result in \eqref{prop3} yields that piecewise polynomial boundary data result in piecewise polynomial extensions of the same degree. The results in \eqref{estext1}, \eqref{estext2} measure smoothness properties of the extension in terms of the smoothness of the boundary data.
\subsection{Global isoparametric mapping} \label{sectdefiso} 
We use the local mappings  $\thetahGamma$, $\Psi^\Gamma$ on $\OGamma$ and the boundary parametrization $\chi_b$ as in \eqref{boundparam}, and combine these with the extension technique to obtain global mappings $\thetah$ and $\Psi$. For this we decompose the triangulation $\T$ into disjoint subsets as follows
\[
 \T=\T^\Gamma\cup \T_1 \cup \T_2, \quad \text{with}~~\T_1:=\{\, T\in \T~|~(\hat\phi_h)_{|T} < 0\,\},~~\T_2:=\{\, T\in \T~|~(\hat\phi_h)_{|T} > 0\,\}.
\]
We assumed that $\overline{\Omega}_1$ (where $\phi <0$ holds) is strictly contained in $\Omega$. Hence (for $h$ sufficiently small) we have that $\T_1$ is strictly contained in $\Omega$. On $\partial \T_1$ we have boundary data $\thetah$ and $\Psi^\Gamma$, which will be extended to $\T_1$, cf. below. The boundary of $\T_2$ consists of two disjoint parts, namely  $\partial \T_2= \partial \T \cup (\partial\T_2\setminus \partial \T)$, cf. Figure \ref{fig:figt1t2} (right).
On $(\partial\T_2\setminus \partial \T) \subset \partial \T^\Gamma$ we have boundary data $\thetah$ and $\Psi^\Gamma$. On $\partial\T$ we have boundary data given by the parametrization $\chi_b$.
Below, boundary data on the two parts $\partial\T_2\setminus \partial \T$ and $\partial \T$ of $\partial \T_2$ are denoted as a pair of functions $(g_1,g_2)$. Using the extension technique, the local  mappings  $\Psi^\Gamma$ and $\thetahGamma$ on $\OGamma$ we can define \emph{global} mappings $\Psi$, $\thetah$ as follows:
% \begin{equation} \begin{split}
%     \Psi & := \left\{ \begin{array}{rcl}
%                               \Psi^\Gamma & \text{ on} & \T^\Gamma \\
%                               \id +\mathcal{E}_k\big(\T_1,(\Psi^\Gamma -\id)\big) & \text{ on} & \T_1 \\
%                                \id +\mathcal{E}_k\big(\T_2,(\Psi^\Gamma -\id, \chi_b)\big) & \text{ on} & \T_2, \\
%                               \end{array}\right. \\
%  \thetah &:= \left\{ \begin{array}{rcl}
%                               \thetahGamma & \text{ on} & \T^\Gamma \\
%                               \id +\mathcal{E}_k\big(\T_1,(\thetahGamma -\id)\big) & \text{ on} & \T_1 \\
%                                \id +\mathcal{E}_k\big(\T_2,(\thetahGamma -\id, I_k\chi_b)\big) & \text{ on} & \T_2. \\
%                               \end{array}\right.  \label{deftheta}
% \end{split}
% \end{equation}
\begin{subequations}\label{deftheta}
\begin{align}
    \Psi & := \left\{ \begin{array}{r@{~}lcl}
                              \Psi^\Gamma & & \text{ on} & \T^\Gamma\!\!\!, \\
                              \id &+\mathcal{E}_k\big(\T_1,(\Psi^\Gamma -\id)\big) & \text{ on} & \T_1, \\
                               \id &+\mathcal{E}_k\big(\T_2,(\Psi^\Gamma_{\hphantom{h}} -\id, \hphantom{\chi_b} \chi_b)\big) & \text{ on} & \T_2,
                              \end{array}\right. \label{defpsi}\\
 \thetah &:= \left\{ \begin{array}{r@{~}lcl}
                       \thetahGamma &\hphantom{+\mathcal{E}_k\big(\T_1,(\Psi^\Gamma -\id)\big)}& \text{ on} & \T^\Gamma\!\!\!, \\
                              \id &+\mathcal{E}_k\big(\T_1,(\thetahGamma -\id)\big) & \text{ on} & \T_1, \\
                               \id &+\mathcal{E}_k\big(\T_2,(\thetahGamma -\id, I_k\chi_b)\big) & \text{ on} & \T_2.
                              \end{array}\right.  \label{defthetah}
\end{align}
\end{subequations}
Here $I_k\chi_b$ denotes the piecewise polynomial nodal interpolation  of the boundary parametrization $\chi_b$ on the edges ($d=2$) or faces ($d=3$) that are on $\partial \T$.

\begin{figure}[h!]
  \vspace*{-0.2cm}
  \begin{center}
    \includegraphics[width=0.98\textwidth]{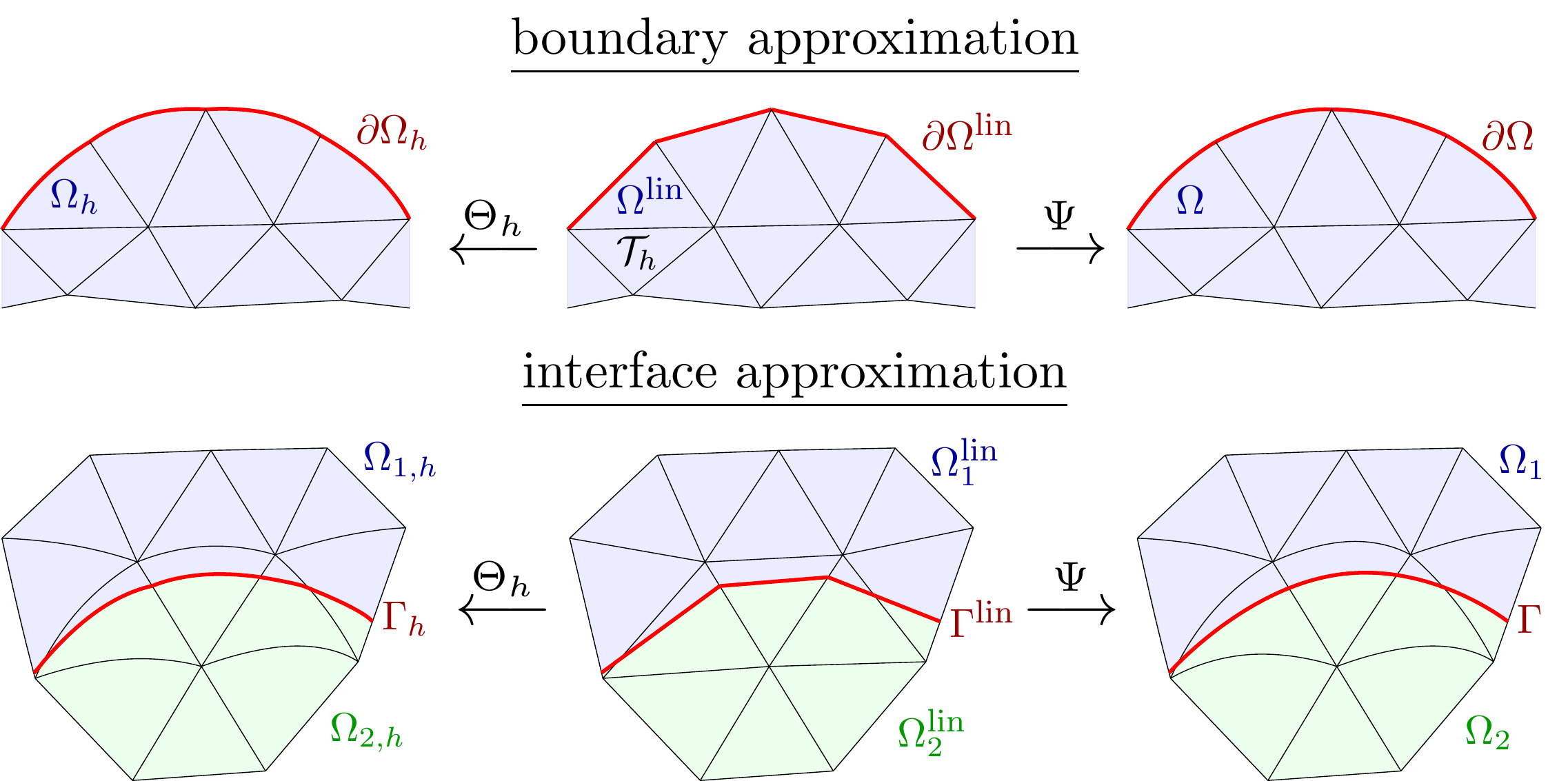} 
  \end{center}
  \vspace*{-0.2cm}
  \caption{Sketch of geometries. The basis is the polygonal approximation (center) of the boundary (top row) and the interface (bottom row). The parametric mapping $\Theta_h$ (left) is used in the discretization method and essential for the higher order accuracy. The mapping $\Psi$ (right) maps the discrete boundary and interface approximations to the exact geometries and is used in the error analysis below.}
  \label{fig:trafos} 
\end{figure}

\subsection{Isoparametric unfitted finite element method} \label{sectmethod}
We introduce the isoparametric unfitted finite element method, based on the isoparametric mapping $\Theta_h$, for the discretization
of  the model elliptic interface problem \eqref{eq:ellmodel}.
We define the isoparametric Nitsche unfitted FEM  as a transformed version of the original Nitsche unfitted FE discretization \cite{hansbo2002unfitted} with respect to the interface approximation $\Gamma_h = \Theta_h(\Gammalin)$.  We introduce some further notation.
The standard unfitted space w.r.t. $\Gammalin$ is denoted by
\[
  V_h^\Gamma:= {V_h^k}_{|\Omega^{\lin}_{h,1}}\oplus {V_h^k}_{|\Omega^{\lin}_{h,2} }.  
\]
To simplify the notation we do not explicitly express the polynomial degree $k$ in $V_h^\Gamma$. 
The \emph{isoparametric unfitted FE space} is defined as
\begin{equation}\label{transfspace}
 V_{h,\Theta}^\Gamma:= \{\, v_h \circ \Theta_h^{-1}~|~ v_h \in V_h^\Gamma\, \}= \{\, \tilde v_h~|~\tilde v_h \circ \Theta_h \in  V_h^\Gamma\, \}.
\end{equation} 
Note that functions from $V_{h,\Theta}^\Gamma$ are defined on the isoparametrically transformed domain $\Omega_h= \Omega_{1,h} \cup \Omega_{2,h} $, with $\Omega_{i,h}:= \Theta_h(\Omega^{\lin}_{i})$.
Based on this space we formulate a discretization of \eqref{eq:ellmodel} using the Nitsche technique \cite{hansbo2002unfitted} with  $\Omega_{i,h}$, $i=1,2$, as numerical approximation of the subdomains $\Omega_i$: determine $ u_h \in V_{h,\Theta}^\Gamma$ such that
\begin{equation} \label{Nitsche1}
 A_h(u_h,v_h) := a_h(u_h,v_h) + N_h(u_h,v_h) = f_h(v_h) \quad \text{for all } v_h \in V_{h,\Theta}^\Gamma,
\end{equation}
with the bilinear and linear forms
\begin{subequations} \label{eq:blfs}
\begin{align}
a_h(u,v) & := \sum_{i=1}^2 \alpha_i \int_{\Omega_{i,h}} \nabla u \cdot \nabla v ~dx,
%\quad f(v) := \int_{\Omega_h} f^e v~dx, \label{jk} 
\\
N_h(u,v) & := N_h^c(u,v) + N_h^c(v,u) + N_h^s(u,v),\\
N_h^c(u,v) & := \int_{\Gamma_h} \average{-\alpha \nabla v} \cdot n \spacejump{u}~ds, \quad 
N_h^s(u,v) := \bar \alpha \frac{\lambda}{h} \int_{\Gamma_h} \spacejump{u} \spacejump{v}~ds,
\end{align}
\end{subequations}
for $u, v \in  V_{h,\Theta}^\Gamma + \Vreg$ with $\Vreg := H^1(\Omega_h) \cap (H^2(\Omega_{1,h}) \cup H^2(\Omega_{2,h}))$. 
Furthermore, $ n = n_{\Gamma_h}$ denotes the outer normal on the boundary $\Gamma_h$ of $\Omega_{1,h}$ and 
$\bar \alpha = \frac12( \alpha_1 + \alpha_2)$ the mean diffusion coefficient. 
For the averaging operator  $\average{\cdot}$ there are different possibilities. We use 
 $\averageleft w \averageright := \kappa_1 w_{|\Omega_{1,h}} + \kappa_2w_{|\Omega_{2,h}} $ with a ``Heaviside'' choice where $\kappa_1 + \kappa_2 =1$ with $\kappa_1 = 1$ if $|T_1| > \frac12 |T|$ and $\kappa_1 = 0$ otherwise. Here, $T_i = T \cap \Omega^{\lin}_{i}$, i.e. the cut configuration on the undeformed mesh is used. This choice in the averaging renders the scheme in \eqref{Nitsche1} stable (for sufficiently large $\lambda$) for arbitrary polynomial degrees $k$, independent of the cut position of $\Gamma$ (Lemma~5.2 in \cite{CLARH1}). A different choice for the averaging which also results in a stable scheme is $\kappa_i = |T_i|/|T|$.

 In order to define  the  right hand side functional $f_h$ we first assume that the source term $f_i :\, \Omega_i \to \Bbb{R}$ in \eqref{eq:ellmodel1} is (smoothly) extended to $\Omega_{i,h}$. This extension is denoted by $f_{i,h}$ and satisfies $f_{i,h}=f_i$ on $\Omega_i$.
 We define $f_h$ on $\Omega$ by ${f_h}_{|\Omega_{i,h}}:= f_{i,h}$, $i=1,2$, hence,
\begin{equation} \label{jk}
  f_h(v) := \int_{\Omega_h} f_h v\, dx = \sum_{i=1,2} \int_{\Omega_{i,h}} f_{i,h} v dx.
\end{equation}

\begin{remark} \rm 
For the implementation of this method we use the same technique as in standard isoparametric finite element methods. In the integrals we apply a transformation of variables $y:=\Theta_h^{-1}(x)$. For example, the bilinear form $a_h(u,v)$ then results in
\begin{equation} \label{atrans}
 a_h(u,v)  := \sum_{i=1,2} \alpha_i \int_{\Omega^{\lin}_{i}}  D\Theta_h^{-T} \nabla u \cdot  D\Theta_h^{-T} \nabla v ~ \det (D\Theta_h)\, dy.
\end{equation}
Based on  this transformation the implementation of integrals is carried out as for the case of the piecewise planar interface $\Gammalin$. The additional variable coefficients $D \thetah^{-T}$, $\det(D \thetah)$  are easily and efficiently computable using the property that $\thetah$ is a finite element (vector) function. 
\end{remark}
%\newpage
\section{Preliminaries for the error analysis}     \label{sectOld}
In this section we collect results from \cite{CLARH1} that we need in the proof of the $L^2$-estimate in Section~\ref{sec:erroranalysis}. In \cite{CLARH1} the case \emph{without} an isoparametric boundary approximation is considered. It turns out that the results that  we need from \cite{CLARH1} also hold for the case with isoparametric boundary approximation, i.e., with the mapping $\thetah$ in \eqref{defthetah}, and that the proofs given in \cite{CLARH1} require only minor modifications. In the proofs below we refer to the analysis given in \cite{CLARH1} and explain only the modifications needed due to the additional isoparametric boundary approximation. 

In the error analysis we use the norm
\begin{align}
 \Vert v \Vert_{h}^2 & := \vert v \vert_{1}^2 + \Vert \jump{v} \Vert_{\frac12,h,\Gamma_h}^2 + \Vert \average{\alpha \nabla v} \Vert_{-\frac12,h,\Gamma_h}^2, \\
\text{ with } \Vert v \Vert_{\pm \frac12,h,\Gamma_h}^2 & := \left(\bar{\alpha} / h \right)^{\pm 1} \Vert v \Vert_{L^2(\Gamma_h)}^2 \text{ and } |v|_1^2:= \sum_{i=1,2} \alpha_i \Vert \nabla v \Vert_{L^2(\Omega_{i,h})}^2. \label{ppp}
\end{align}
Note that the norms are formulated with respect to $\Omega_{i,h} = \thetah(\Omega^{\lin}_{i})$ and  $\Gamma_{h} = \thetah(\Gammalin)$ and include a scaling depending on $\alpha$.
 %
%We further define the space of sufficiently smooth functions which allow for the evaluation of normal gradients at the interface $\Gamma_h$: 
%\begin{equation} \label{defVreg}
%\Vreg := H^1(\Omega) \cap H^2(\Omega_{1,h} \cup \Omega_{2,h}).
%\end{equation}
The norm $\|\cdot\|_h$ and the bilinear forms in \eqref{eq:blfs} are well-defined on $V_{h,\Theta}^\Gamma +\Vreg $. 

\begin{lemma}[Ellipticity and continuity] \label{lemcoercive}
  For $\lambda $ sufficiently large the estimates
  \begin{subequations}
\begin{align}
  A_h(u,u) &\gtrsim \Vert u \Vert_h^2 && \hspace*{-2cm} \text{for all}~\, u \in V_{h,\Theta}^\Gamma, \label{eq:coerc} \\
  A_h(u,v) &\lesssim \Vert u \Vert_h \Vert v \Vert_h && \hspace*{-2cm} \text{for all}~\, u,v \in V_{h,\Theta}^\Gamma +\Vreg \label{eq:bound}
\end{align}
\end{subequations}
hold.
\end{lemma}
\begin{proof}
The proof given in  \cite[Lemma 5.2]{CLARH1} applies without any modifications. We note that a crucial component in the proof is the assessment that $L^2$ and $H^1$ norms on $\Omega_{i,h}$ and $\Gamma_h$ are equivalent to corresponding norms of mapped functions on $\Omega_{i}^{\lin}$ and $\Gamma^{\lin}$, cf. \cite[Lemma 3.12]{CLARH1}. Due to this the ellipticity and continuity can equivalently be shown on a piecewise linear interface approximation with $\Gammalin$. 
\end{proof}
\ \\[1ex]
In the remainder we assume that $\lambda$ is taken sufficiently large such that the results in Lemma~\ref{lemcoercive} hold. Then the discrete problem \eqref{Nitsche1} has a unique solution. 

A key component  of the error analysis in \cite{CLARH1} is the mapping $\Phi_h:=\Psi \circ \thetah^{-1}$. For the derivation of properties of this mapping we first derive useful results for $\Psi$.

  \begin{theorem} \label{ThmPsi}
   For $\Psi$ the following holds:
   \begin{align} 
    \|\Psi - \id\|_{\infty,\Omega}+ h\|D\Psi - I\|_{\infty,\Omega} & \lesssim h^2, \label{bpsi1} \\ 
    \max_{T\in\T} \Vert D^l \Psi \Vert_{\infty,T} & \lesssim 1, \quad 0 \leq l \leq k+1. \label{bpsi2}
   \end{align}
  \end{theorem}
  \begin{proof}
  On $\TGamma$  the results are a direct consequence of Lemma~\ref{lem:boundpsig}. On  $\T_1 \cup \T_2$ the result in \eqref{bpsi2}  for $l=0,1$ follows from \eqref{bpsi1} and for $l \geq 2$ it follows from \eqref{estext2} combined with \eqref{fistderPsiA} and \eqref{eq:assumptionchi}. It remains to derive the estimate as in \eqref{bpsi1} on $\T_1 \cup \T_2$. We consider $\T_2$ and write $\E(\Psi^\Gamma -\id, \chi_b):=\mathcal{E}_k\big(\T_2,(\Psi^\Gamma -\id, \chi_b)\big)$. Using \eqref{estext1} and $\chi_b(x_i)=0$ at vertices $x_i \in \partial \T$ we get
 \begin{align*}
   \|\Psi
   - \id \|_{\infty,\T_2} &+ h\|D\Psi - I\|_{\infty,\T_2} \\
 = & \|\mathcal{E}(\Psi^\Gamma -\id, \chi_b)\|_{\infty,\T_2}+ h \|D\mathcal{E}(\Psi^\Gamma -\id, \chi_b)\|_{\infty,\T_2} \\
    \stackrel{\eqref{estext1}}{\lesssim} &  \sum_{r=0}^{k+1} h^{r} \Big( \max_{F \in \F(\partial \T_2 \setminus \partial \T)}\|D^r (\Psi^\Gamma-\id)\|_{\infty,F}+ \max_{F \in \F(\partial \T)}\|D^r \chi_b\|_{\infty,F} \Big) 
 \\ & +  \max_{x_i \in \V(\partial \T_2 \setminus \partial \T)}|(\Psi^\Gamma-\id)(x_i)| \\
   \lesssim
   &
     \underbrace{\|\Psi^\Gamma - \id\|_{\infty,\TGamma}+h  \|D\Psi^\Gamma - I\|_{\infty,\TGamma}}_{\lesssim h^2 ~~ (\text{Lemma}~\ref{lem:boundpsig})}
     + h^2 \underbrace{\max_{2 \leq l \leq k+1} \max_{T \in \TGamma} \|D^l \Psi^\Gamma\|_{\infty,T}}_{\lesssim 1 ~~ (\text{Lemma}~\ref{lem:boundpsig})}
  \\ & + \underbrace{\sum_{r=0}^{k+1} h^{r}\max_{F \in \F(\partial \T)}\|D^r \chi_b\|_{\infty,F}}_{\lesssim h^2 ~~ \eqref{eq:assumptionchi}}+ \underbrace{\max_{x_i \in \V(\TGamma)}|d(x_i)|}_{\lesssim h^2 ~~ (\text{Lemma}~\ref{lem:boundpsig})} \lesssim h^2.
   \end{align*}
For $\T_1$ similar arguments can be applied. 
  \end{proof}\\[1ex]
The proof above generalizes the one given in \cite{CLARH1} since it allows a boundary approximation via the function $g_b(x)=x+ \chi_b(x)$. 

 From \eqref{bpsi1} it follows that, for $h$ sufficiently small, $\Psi$ is a bijection on $\Omega$. Furthermore this mapping induces a family of (curved) finite elements that is regular of order $k$, in the sense as defined in \cite{bernardi1989optimal}. The corresponding  curved finite element space is given by 
  \begin{equation} \label{FEcurvedPsi}
    V_{h,\Psi}:= \{\, v_h \circ \Psi_h^{-1}~|~ v_h \in V_h^k\, \}.
  \end{equation}
Due to the results in Theorem~\ref{ThmPsi} the analysis of the approximation error for this finite element space as developed in \cite{bernardi1989optimal} can be applied. Corollary 4.1 from that paper yields that there exists an interpolation operator $\Pi_h:\, H^{k+1}(\Omega) \to  V_{h,\Psi}$ such that
\begin{equation} \label{errorBernardi}
 \|u- \Pi_h u\|_{L^2(\Omega)}+ h\|u-\Pi_h u\|_{H^1(\Omega)} \lesssim h^{k+1} \|u\|_{H^{k+1}(\Omega)} \quad \text{ for all } u \in H^{k+1}(\Omega).
\end{equation}

For $h$ sufficiently small, the mapping $\Phi_h= \Psi \circ \thetah^{-1}: \, \Omega_h \to \Omega$ is a bijection and has the property $\Phi_h(\Gamma_h)=\Gamma$, cf. Fig.~\ref{fig:trafos}. In the remainder we assume that $h$ is sufficiently small such that $\Phi_h$ is a bijection. It has the smoothness property $\Phi_h \in C(\Omega_h)^d \cap C^{k+1}(\Theta_h(\T))^d$.  In the following lemma we derive  further properties of $\Phi_h$ that will be needed in the error analysis.
\begin{lemma} \label{lem:phihbounds}
 The following holds:
  \begin{align}
    \|\thetah - \Psi\|_{\infty,\Omegalin}+ h\|D(\thetah - \Psi)\|_{\infty,\Omegalin} & \lesssim h^{k+1}, \label{estt1} \\
     \|\Phi_h - \id\|_{\infty,\Omegalin}+ h\|D\Phi_h - I\|_{\infty,\Omegalin} & \lesssim h^{k+1}.\label{estt2} 
  \end{align}
\end{lemma}
\begin{proof} The estimate \eqref{estt1} follows by using the linearity of the extension operator $\E_k$, cf.~\eqref{deftheta}, and then using argmuments very similar to the ones used in the proof of Theorem~\ref{ThmPsi}.
Note that $\Phi_h-\id=(\Psi-\thetah)\thetah^{-1}$. From  the results in  \eqref{bpsi1} and \eqref{estt1} it follows that $\|\thetah^{-1}\|_{\infty,\Omega_h} \lesssim 1$, $\|D\thetah^{-1}\|_{\infty,\Omega_h} \lesssim 1$. Hence, the estimate in \eqref{estt2} follows from the one in \eqref{estt1}.
\end{proof}
\\[1ex]

We define $\Vregphi :=H^1(\Omega)\cap  H^2(\Omega_1 \cup \Omega_2) $.
%and $\Vregphih :=H^1(\Omega_h)\cap  H^2(\Omega_{1,h} \cup \Omega_{2,h}) $.
Related to  $\Phi_h$ we define the following transformed unfitted finite element space, cf. \eqref{transfspace},
\[ V_{h,\Phi}^\Gamma := \{ v \circ \Phi_h^{-1}, v \in V_{h,\Theta}^\Gamma\} \subset H^1(\Omega_1 \cup \Omega_2),
\]
and the linear and bilinear forms
\begin{equation}
A(u,v) := a(u,v) + N(u,v), \quad f(v)= \int_{\Omega} f v \, dx, \quad u,v \in V_{h,\Phi}^\Gamma + \Vregphi,
\end{equation}
with the bilinear forms $a(\cdot,\cdot)$, $N(\cdot,\cdot)$ as in \eqref{eq:blfs} with $\Omega_{i,h}$ and $\Gamma_h$ replaced with $\Omega_i$ and $\Gamma$, respectively. In the following lemma a consistency result is given.
% The (bi)linear forms in \eqref{subeq} are also well-defined on $\Vreg+V_h^\Gamma$ if the mapping $\Theta_h$ is replaced by $\Psi$.
%  Doing so, we get the bilinear form $A_h(\Psi;\cdot, \cdot)$, which has the consistency property formulated in the following lemma.
\begin{lemma}[Consistency] \label{lemconsis}
Let $u\in \Vregphi$ be a solution of \eqref{eq:ellmodel}.  The following holds:
\begin{equation} \label{eq:cons1}
  A(u,v) = f(v) \quad \text{for all}~ \, v \in V_{h,\Phi}^\Gamma + \Vregphi.
\end{equation}
\end{lemma}
\begin{proof} Proof is given in \cite{CLARH1}. No modifications are needed.
\end{proof}
\\[1ex] 
We introduce the subdomain
\[
 U_{\delta}:=\{\, x\in \Omega~|~{\rm dist}(x,\Gamma) \leq \delta \text{ or } {\rm dist}(x,\partial \Omega) \leq \delta \,\},
\]
with $\delta > 0$ (sufficiently small), consisting of a tubular neigborhood of the interface $\Gamma$ and  a strip adjacent to $\partial \Omega$.
To bound the consistency errors with respect to $A_h$ we use the following result.
\begin{lemma} \label{lem:elliott}
For $\delta > 0$ sufficiently small and $u \in H^1(\Omega_{1} \cup \Omega_{2})$ there holds
\begin{equation}\label{eq:elliott}
  \Vert u \Vert_{L^2(U_\delta)} \leq c \delta^{\frac12} \Vert u \Vert_{H^1(\Omega_{1}\cup\Omega_{2})}.
\end{equation}
\vspace*{-0.25cm}
\end{lemma}
\begin{proof}
See \cite[Lemma 4.10]{elliott2012finite}.
\end{proof} \\[1ex]
With this result a consistency bound from \cite{CLARH1} can be improved.
\begin{lemma}[Consistency bounds] \label{lemconsist}
Let $u \in  \Vregphi$ be a solution of \eqref{eq:ellmodel}. We assume $f\in H^{1,\infty}(\Omega_1 \cup \Omega_2)$ and a data extension $f_h$, used in \eqref{jk}, that satisfies $\|f_h\|_{H^{1,\infty}(\Omega_{1,h}\cup \Omega_{2,h})} \lesssim \|f\|_{H^{1,\infty}(\Omega_1\cup \Omega_2)}$. The following estimates hold for $w_h \in V_{h,\Theta}^\Gamma + \Vregphih$:
\begin{subequations}
\begin{align}
 |A(u,w_h\circ \Phi_h^{-1}) - A_h(u\circ \Phi_h,w_h)| & \lesssim h^k \Vert u \Vert_{H^2(\Omega_1 \cup \Omega_2)} \Vert w_h \Vert_h \label{p1},
\\
 |f(w_h\circ \Phi_h^{-1})-f_h(w_h)| & \lesssim h^{k+1} \|f\|_{H^{1,\infty}(\Omega_1\cup \Omega_2)} \| w_h \Vert_h. \label{p2}
%\\
% Furthermore, for $w \in \Vreg$, there holds}
%  |A(u,w) - A_h(u\circ \Phi_h,w\circ \Phi_h)| & \lesssim h^{k+1} \Vert u \Vert_{H^2(\Omega_1 \cup \Omega_2)}
% \Vert w \Vert_{H^2(\Omega_1 \cup \Omega_2)}. \label{p3}
\end{align}
% {\bf NB: I put $h^{k+1}$ in the result ineq:theta} \eqref{p2}. I think, that the proof in the previous paper almost immediately gives this bound; check!}
\end{subequations}
 \end{lemma}
\begin{proof}
In \cite[Lemma 5.13]{CLARH1} equation \eqref{p1} is derived for the case without isoparametric boundary approximation. The proof applies, without modifications, also to the case with isoparametric boundary approximation. In \cite[Lemma 5.13]{CLARH1} a bound as in \eqref{p2} with $h^{k+1}$ replaced by $h^k$ is derived. We now show how one obtains the improved 
bound in  \eqref{p2}.
Define $\Delta_h:=U_{\delta_h}$, which $\delta_h:= ch$ and $c>0$ sufficiently large such that $\Phi_h^{-1} \neq \id$ only on $\Delta_h$. Note that
\begin{align*}
  | f(w_h \circ \Phi_h^{-1})-  f_h(w_h)|& = \left| \int_{\Omega} f (w_h \circ \Phi_h^{-1}) \, dx - \int_{\Omega_h} f_h w_h \, dx \right| \\
  & \leq \sum_{i=1}^2 \left| \int_{\Omega_i} f_i (w_h \circ \Phi_h^{-1}) \, dx - \int_{\Omega_{i,h}} f_{i,h} w_h \, dx \right|.
\end{align*}
For $i=1,2$ and using $f_{i,h}=f_i$ on $\Omega_{i}$ we get:
\begin{align*}
 &  \left| \int_{\Omega_i} f_i (w_h \circ \Phi_h^{-1}) \, dx - \int_{\Omega_{i,h}} f_{i,h} w_h \, dx \right|  = \left| \int_{\Omega_{i,h}} \big(\det(D\Phi_h) (f_{i,h} \circ \Phi_h) - f_{i,h} \big) w_h ~ dx \right|  \\
& \leq  \int_{\Delta_h \cap \Omega_{i,h}} \left| (\det(D\Phi_h) - I) f_{i,h} w_h \right|  dx 
+  \int_{\Delta_h\cap \Omega_{i,h}} \left| \det(D\Phi_h) ( (f_{i,h} \circ \Phi_h) - f_{i,h} ) w_h\right| dx  \\
& \leq |\Delta_h|^{\frac12} ~ \Vert \det(D\Phi_h) - I \Vert_{\infty,\Omega_h} \Vert f_{h} \Vert_{\infty,\Omega_h}  \Vert w_h 
\Vert_{L^2(\Delta_h)}\\
& \quad + |\Delta_h|^{\frac12} ~ \Vert \det(D\Phi_h)\Vert_{\infty,\Omega_h} 
\Vert f_{h} \Vert_{H^{1,\infty}(\Omega_{1,h} \cup \Omega_{2,h})} \Vert \Phi_h - \id \Vert_{\infty,\Omega_h} \Vert w_h \Vert_{L^2(\Delta_h)} \\
& \lesssim h^{k+\frac12} \Vert f \Vert_{H^{1,\infty}(\Omega_{1,h}\cup \Omega_{2,h})} \Vert w_h \Vert_{L^2(\Delta_h)}.
\end{align*}
 In the last inequality we used the properties of the transformation $\Phi_h$ from \eqref{estt2}.
Next, we apply Lemma \ref{lem:elliott} to gain another factor $h^{\frac12}$:
\begin{equation} \label{eq:gainfactor}
  \begin{split}
    \Vert w_h \Vert_{L^2(\Delta_h)} & \lesssim \Vert w_h \circ \Phi_h^{-1} \Vert_{L^2(\Phi_h(\Delta_h))} 
    \leq \Vert w_h \circ \Phi_h^{-1} \Vert_{L^2(U_{\delta_h})} \\
    & \lesssim \delta_h^{\frac12} \Vert w_h \circ \Phi_h^{-1} \Vert_{H^1(\Omega_1\cup\Omega_2)}
    \lesssim h^{\frac12} \Vert w_h \Vert_{H^1(\Omega_{1,h}\cup\Omega_{2,h})}
    \lesssim h^{\frac12} \Vert w_h \Vert_{h}.
  \end{split}
\end{equation}
Combining these results completes the proof. 
%Further we used that $D\Phi_h$ gets arbitrarily close to $I$ for $h$ sufficiently small. And finally a Poincar\'e-type inequality to bound the $H^1(\Omega_{1,h}\cup\Omega_{2,h})$ norm by $\Vert \cdot \Vert_h$ is applied.
\end{proof}
\\[1ex]
The main result of \cite{CLARH1} is given in the following theorem. 
\begin{theorem}[Discretization error bound] \label{mainthmH1}
 Let $u$ be the solution of \eqref{eq:ellmodel} and $u_h \in V_{h,\Theta}^\Gamma$  the solution of \eqref{Nitsche1}.We assume that  $u\in H^{k+1}(\Omega_1 \cup \Omega_2)$ and the data extension $f_h$ satisfies the condition in Lemma~\ref{lemconsist}.  Then the following holds:
\begin{align}\label{main1}
\Vert u\circ \Phi_h - u_h \Vert_h & \lesssim h^k( \| u \|_{H^{k+1}(\Omega_1 \cup \Omega_2)} + \|f\|_{H^{1,\infty}(\Omega_1\cup\Omega_2)}),
  \\
|u- u_h \circ \Phi_h^{-1}|_{H^1(\Omega_1 \cup \Omega_2)} &\lesssim h^k( \| u \|_{H^{k+1}(\Omega_1 \cup \Omega_2)}+ \|f\|_{H^{1,\infty}(\Omega_1\cup \Omega_2)}) \label{cor1}.
\end{align}
\end{theorem}
\begin{proof}
The result \eqref{cor1} easily follows from \eqref{main1} using the definition of the norm $\|\cdot\|_h$ and properties of $\Phi_h$. 
 The proof of \eqref{main1} in \cite{CLARH1} applies with only minor modifications. A key ingredient in the proof is the following approximation property, which holds for arbitrary $u\in H^{k+1}(\Omega_1 \cup \Omega_2)$:
\begin{equation} \label{apppr}
\inf_{v_h \in V_{h,\Theta}^\Gamma} \|u \circ \Phi _h - v_h\|_h \lesssim h^k \|u\|_{H^{k+1}(\Omega_1 \cup \Omega_2)},
 \end{equation}
 which is based on the isoparametric interpolation error bound \eqref{errorBernardi}. 
\end{proof}

\section{$L^2$-error bound}\label{sec:erroranalysis}
In this section we present an $L^2$-norm discretization error bound for the unfitted finite element method in \eqref{Nitsche1}. 
The key ingredients are a duality argument, the already available $H^1$-error bound given in Theorem~\ref{mainthmH1} and the consistency error bound \eqref{p1}-\eqref{p2}.

In the remainder we assume that the smoothness conditions on the solution $u$ and the right hand side $f$ as formulated in Theorem~\ref{mainthmH1} are satisfied. We define the discretization error $e_h:=u- u_h\circ \Phi_h^{-1} \in H^1(\Omega_1 \cup \Omega_2)$ and consider the (adjoint) problem:
\begin{subequations} \label{eq:ellmodeladj}
\begin{align}
- \mathrm{div} (\alpha_i \nabla {z}) &= \,e_h  
\quad \text{in}~~ \Omega_i , ~i=1,2, \label{eq:ellmodel1adj} \\
\spacejump{{\alpha} \nabla {z} }_{\Gamma} \cdot n_\Gamma &= \, 0, \quad \spacejump{{z}}_{\Gamma} = 0 \quad \text{on}~~\Gamma, \label{eq:ellmodel2adj} \\
z &= 0 \quad \text{ on } \partial \Omega.
 \end{align}
\end{subequations}
We assume that $\Gamma$ and $\partial \Omega$ are sufficiently smooth such that the following regularity property holds:
\begin{equation} 
 \|z\|_{H^2(\Omega_1 \cup \Omega_2)} \lesssim \|e_h\|_{L^2(\Omega)},  \label{reg11} \\
\end{equation}
cf.~\cite{Huang2002,Graham2010,ladyzhenskaya1973linear}.  From Lemma~\ref{lemconsis} we obtain the identity $\|e_h\|_{L^2(\Omega)}^2 = A(z,e_h)$. The $L^2$-error analysis is based on the following splitting, with $z_h \in V_{h,\Theta}^\Gamma$:
\begin{subequations} \label{terms}
\begin{align}
 \|e_h\|_{L^2(\Omega)}^2 &= A(z,e_h)  \nonumber \\
 &= A(z,e_h)-A_h(z \circ \Phi_h, e_h \circ \Phi_h) \label{term1} \\
 & \quad + A_h(z \circ \Phi_h- z_h, e_h \circ \Phi_h) \label{term2} \\
 & \quad + A_h(u\circ \Phi_h, z_h)- A(u, z_h \circ \Phi_h^{-1}) \label{term3} \\
 & \quad + f(z_h\circ \Phi_h^{-1})- f_h(z_h). \label{term4}
\end{align}
\end{subequations}
In the lemmas below we derive bounds for the terms in \eqref{term1}-\eqref{term4}.
\begin{lemma} \label{lemT1}
 The estimate
\[
  \left|A(z,e_h)-A_h(z \circ \Phi_h, e_h \circ \Phi_h) \right| \lesssim h^{2k} \big(\|u\|_{H^{k+1}(\Omega_1 \cup \Omega_2)}+\|f\|_{H^{1,\infty}(\Omega_1\cup \Omega_2)}\big) \|e_h\|_{L^2(\Omega)}
\]
holds.
\end{lemma}
\begin{proof}
Note that $e_h \circ \Phi_h \in V_{h,\Theta}^\Gamma + V_{{\rm reg},h}$. We use \eqref{p1} with $w_h= e_h \circ \Phi_h$, $z$ instead of $u$ and combine this with the regulariy estimate \eqref{reg11} and  the result in Theorem~\ref{mainthmH1}. This yields
\begin{align*}
  \left|A(z,e_h)-A_h(z \circ \Phi_h, e_h \circ \Phi_h) \right| & \lesssim h^k \|z\|_{H^2(\Omega_1 \cup \Omega_2)} \|e_h \circ \Phi_h\|_h \\
 & \lesssim h^{2k} \big(\|u\|_{H^{k+1}(\Omega_1 \cup \Omega_2)}+\|f\|_{H^{1,\infty}(\Omega_1\cup \Omega_2)}\big) \|e_h\|_{L^2(\Omega)},
\end{align*}
which completes the proof. 
\end{proof}
\ \\[1ex]
In the terms \eqref{term2}-\eqref{term4} we need a suitable $z_h \in V_{h,\Theta}^\Gamma$. For this we take $z_h \in V_{h,\Theta}^\Gamma$ such that the following holds, cf.~\eqref{apppr}:
\begin{equation} \label{zh}
 \|z \circ \Phi_h -z_h\|_h \lesssim h \|z\|_{H^2(\Omega_1 \cup \Omega_2)} \lesssim h \|e_h\|_{L^2(\Omega)}.
\end{equation}

\begin{lemma} \label{lemT2}
 Let $z_h$ be as in \eqref{zh}.  The estimate 
\[
  \left| A_h(z \circ \Phi_h- z_h, e_h \circ \Phi_h)\right|  \lesssim h^{k+1} \big(\|u\|_{H^{k+1}(\Omega_1 \cup \Omega_2)}+\|f\|_{H^{1,\infty}(\Omega_1\cup \Omega_2)}\big) \|e_h\|_{L^2(\Omega)} 
\]
holds.
\end{lemma}
\begin{proof}
Using the continuity result in \eqref{eq:bound}, the result in Theorem~\ref{mainthmH1} and the estimate \eqref{zh} we obtain
\begin{align*}
 \left| A_h(z \circ \Phi_h- z_h, e_h \circ \Phi_h)\right| & \lesssim h^k\|z \circ \Phi_h- z_h\|_h  \big(\|u\|_{H^{k+1}(\Omega_1 \cup \Omega_2)}+\|f\|_{H^{1,\infty}(\Omega_1\cup \Omega_2)}\big) \\
 & \lesssim h^{k+1} \big(\|u\|_{H^{k+1}(\Omega_1 \cup \Omega_2)}+\|f\|_{H^{1,\infty}(\Omega_1\cup \Omega_2)}\big) \|e_h\|_{L^2(\Omega)} .
\end{align*}
\end{proof}
\ \\[1ex]
\begin{lemma} \label{lemT3}
 Let $z_h$ be as in \eqref{zh}. The estimate 
\begin{equation} \label{kk}
  |A_h(u\circ \Phi_h, z_h)- A(u, z_h \circ \Phi_h^{-1})| \lesssim h^{k+1} \|u\|_{H^2(\Omega_1 \cup \Omega_2)}  \|e_h\|_{L^2(\Omega)}
\end{equation}
holds.
\end{lemma} 
\begin{proof}
We use the splitting
\begin{align}
 & A_h(u\circ \Phi_h, z_h)- A(u, z_h \circ \Phi_h^{-1}) \nonumber \\
& = A_h(u\circ \Phi_h, z_h- z\circ \Phi_h)- A(u, z_h \circ \Phi_h^{-1}- z) \label{ll1} \\
 & + A_h(u\circ \Phi_h,z\circ \Phi_h)- A(u,z).\label{ll2}
\end{align}
For the term in \eqref{ll1} we use the approximation result \eqref{p1} and \eqref{zh}:
\begin{align*}
 & |A_h(u\circ \Phi_h, z_h- z\circ \Phi_h)- A(u, z_h \circ \Phi_h^{-1}- z)| \lesssim h^k\|u\|_{H^2(\Omega_1 \cup \Omega_2)}\|z_h- z\circ\Phi_h\|_h \\
 & \lesssim  h^{k+1}\|u\|_{H^2(\Omega_1 \cup \Omega_2)} \|e_h\|_{L^2(\Omega)}.
\end{align*}
For the term in \eqref{ll2} we use $\jump{u}=\jump{z}=0$ and thus get:
\begin{align*}
  & A_h(u\circ \Phi_h,z\circ \Phi_h)- A(u,z) \\ & = \sum_{i=1}^2 \alpha_i \Big( \int_{\Omega_{i,h}} \nabla(u\circ \Phi_h) \cdot \nabla(z\circ \Phi_h)\, dx - \int_{\Omega_i} \nabla u \cdot \nabla z \, dx \Big) \\
& = \sum_{i=1}^2 \alpha_i \int_{\Omega_i} C \nabla u \cdot \nabla z \, dx,
\end{align*}
with the matrix $C:=\det(B^{-1})B^TB -I,~B:=D\Phi_h$. We have $C \neq 0$ only on $\Delta_h:=U_{\delta_h}$, with $\delta_h =ch$, for suitable $c >0$. Furthermore, $\|C\|_{\infty,\Omega} \lesssim h^k$ holds, cf.~\cite[Proof of Lemma 5.6]{CLARH1}. Using this and  Lemma~\ref{lem:elliott} we obtain
\begin{align*}
 |A_h(u\circ \Phi_h,z\circ \Phi_h)- A(u,z)| & \lesssim h^k \|u\|_{H^1(\Delta_h)}\|z\|_{H^1(\Delta_h)}  \\ & \lesssim h^{k+1} \|u\|_{H^2(\Omega_1\cup \Omega_2)}  \|z\|_{H^2(\Omega_1\cup \Omega_2)} \\
  & \lesssim  h^{k+1} \|u\|_{H^2(\Omega_1\cup \Omega_2)}\|e_h\|_{L^2(\Omega)}.
\end{align*}
Combing these results completes the proof.
\end{proof}
\ \\[1ex]
\begin{lemma} \label{lemT4}
 Let $z_h$ be as in \eqref{zh}. The estimate
\begin{equation} \label{ll}
  |f(z_h\circ \Phi_h^{-1})-f_h(z_h)|  \lesssim h^{k+1} \|f\|_{H^{1,\infty}(\Omega_1\cup \Omega_2)}\|e_h\|_{L^2(\Omega)} \end{equation}
holds.
\end{lemma}
\begin{proof}
 Using \eqref{p2} we get
\[
 |f(z_h\circ \Phi_h^{-1})-f_h(z_h)\|  \lesssim h^{k+1} \|f\|_{H^{1,\infty}(\Omega_1\cup \Omega_2)} \|z_h\|_h
\]
Furthermore, using \eqref{zh} we get 
\[
 \|z_h\|_h  \leq \|z_h-z\circ \Phi_h\|_h + \|z\circ \Phi_h\|_h   \lesssim h\|e_h\|_{L^2(\Omega)}  +\|z\|_{H^2(\Omega_1 \cup \Omega_2)} 
  \lesssim \|e_h\|_{L^2(\Omega)}.
\]
Thus we obtain the  bound in \eqref{ll}.
\end{proof}
\ \\[2ex]
Combining the previous results  we obtain the following main result.
\begin{theorem}[$L^2$-error bound] \label{mainthmL2} We assume that the smoothness conditions on the solution $u$ and on the data $f,f_h$ as formulated in Theorem~\ref{mainthmH1} are satisfied and that $\Gamma$ and $\partial \Omega$ are sufficiently smooth such that we have the regularity property \eqref{reg11} for the solution of the problem \eqref{eq:ellmodeladj}. The following bound holds for the discretization error: 
\begin{equation}\label{mainl2}
\begin{split}
\Vert u - u_h \circ \Phi_h^{-1} \Vert_{L^2(\Omega)} & \simeq \Vert u \circ \Phi_h - u_h \Vert_{L^2(\Omega_h)}  \\ & \lesssim h^{k+1} (\|u\|_{H^{k+1}(\Omega_1 \cup \Omega_2)} + \|f\|_{H^{1,\infty}(\Omega_1\cup \Omega_2)}).
\end{split} \end{equation}
\end{theorem}

\section{Numerical experiment} \label{sec:numex}
Results of numerical experiments that confirm the optimal $h^{k+1}$ convergence in the $L^2$-norm for the interface problem \eqref{eq:ellmodel} are given in \cite{lehrenfeld15,CLARH1}. However, In these examples the domain boundary is polygonal. In the experiment presented in this section we consider a case where the \emph{curved} domain boundary is approximated by isoparametric elements.

The domain is a disk around the origin with radius $R=2$, $\Omega = B_2((0,0))$. The domain is divided into an inner disk $\Omega_1=B_1((0,0))$ and an outer ring $\Omega_2=\Omega\setminus\Omega_1$. Hence, the interface is a circle described with the level set function $\phi(x) = \Vert x \Vert_2 -1$,  $\Gamma = \{ \phi(x) = 0 \}$. We approximate $\phi$ by  $\phi_h \in V_h^k$ by interpolation.
We set $(\alpha_1,\alpha_2) = (\pi,2)$ and take the right hand-side  $f$ such that the exact solution is given by (with $r(x) = \Vert x \Vert_2$)
%% {\bf AR: what is $\beta$?}
\begin{equation} \label{eq:numexu}
  u(x) = \left\{ \begin{array}{rl} 2+ (\cos(\frac{\pi}{2}r(x))-1) & x\in\Omega_1, \\ 2 - r(x)   & x\in\Omega_2. \end{array} \right.
\end{equation}
Note that $u$ fulfills the interface conditions. For the Nitsche stabilization parameter we choose $\lambda = 20 k^2$ and measure the $L^2$ error on the curved mesh $\Omega_h$ which is obtained by applying the isoparametric mapping $\Theta_h$.
To measure the errors we use domain-wise the canonical extension of $u$ leading to $u^e$ as in \eqref{eq:numexu} with $\Omega_i$ replaced by $\Omega_{i,h}$.
Starting from a coarse initial mesh, see Figure \ref{fig:numex2d} (top left), we apply three succesive mesh refinements and use polynomial degrees $k=1,..,5$.
To solve the arising linear systems we used a direct solver.
% until either a maximum of 5 mesh refinements has been reached or the error is below $10^{-10}$.
The method has been implemented in the add-on library \texttt{ngsxfem} to the finite element library \texttt{NGSolve} \cite{schoeberl2014cpp11}.

The results are shown in Figure \ref{fig:numex2d}. We observe the predicted optimal $\mathcal{O}(h^{k+1})$ behavior.

\begin{figure}[h!]
\hspace*{0.05\textwidth}
\begin{minipage}{0.35\textwidth}
\begin{center}
\includegraphics[width=0.95\textwidth]{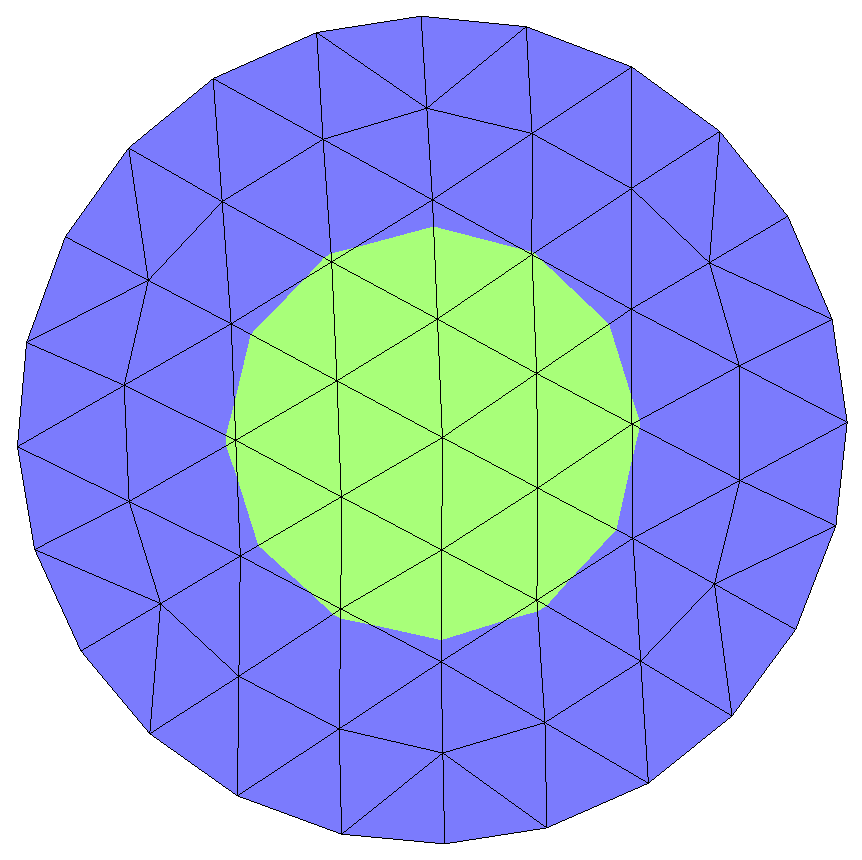} \\[1ex]
\includegraphics[width=0.95\textwidth]{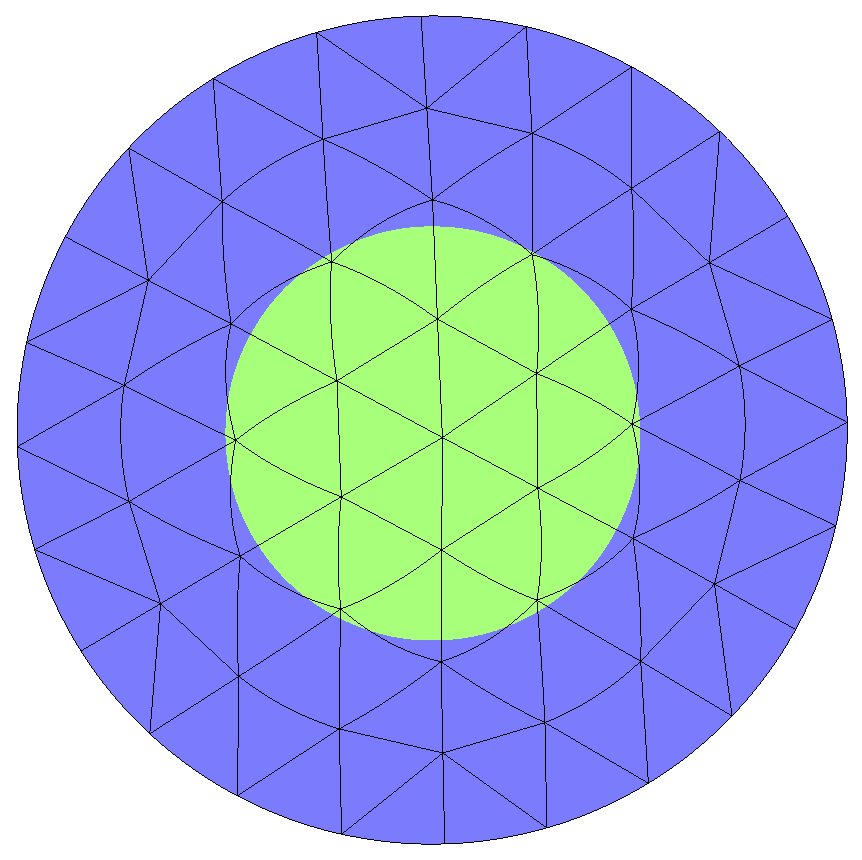}
\end{center}
\end{minipage}
\hspace*{0.12\textwidth}
\begin{minipage}{0.35\textwidth}
\small
\begin{center}
\begin{tabular}{r@{\ \ }r@{\ \ \ }
  l@{\ (}c@{)\ \ \ }
}
$k$ & $L$ & $\Vert u^e - u_h\Vert_{L^2(\Omega_h)}$ & eoc\\
\toprule
1 & 0 & \num{ 7.36800414e-02} &   - \\
  & 1 & \num{ 1.90253353e-02} & 1.95 \\
  & 2 & \num{ 4.72040148e-03} & 2.01 \\
  & 3 & \num{ 1.17707175e-03} & 2.00 \\
  % & 4 & \num{ 2.96982454e-04} & 1.99 \\
  % & 5 & \num{ 7.44175054e-05} & 2.00 \\
  % & 6 & \num{ 1.86119738e-05} & 2.00 \\
  % & 7 & \num{ 4.65297235e-06} & 2.00 \\
\midrule
2 & 0 & \num{  5.53387584e-03} &   - \\
  & 1 & \num{  5.85824788e-04} & 3.24 \\
  & 2 & \num{  6.38353911e-05} & 3.20 \\
  & 3 & \num{  7.70314676e-06} & 3.05 \\
  % & 4 & \num{  9.49054443e-07} & 3.02 \\
  % & 5 & \num{  1.17978345e-07} & 3.01 \\
  % & 6 & \num{  1.47104379e-08} & 3.00 \\
\midrule
3 & 0 & \num{ 6.73584972e-04} &   - \\
  & 1 & \num{ 2.71968655e-05} & 4.63 \\
  & 2 & \num{ 1.39476599e-06} & 4.29 \\
  & 3 & \num{ 7.50056132e-08} & 4.22 \\
  % & 4 & \num{ 4.71746235e-09} & 3.99 \\
  % & 5 & \num{ 2.88734098e-10} & 4.03 \\
\midrule
4 & 0 & \num{ 9.31770985e-05} &   - \\
  & 1 & \num{ 1.28501095e-06} & 6.18 \\
  & 2 & \num{ 2.62992882e-08} & 5.61 \\
  & 3 & \num{ 6.24838787e-10} & 5.40 \\
  % & 4 & \num{ 1.54235418e-11} & 5.34 \\
\midrule
5 & 0 & \num{ 1.82853094e-05} &   - \\
  & 1 & \num{ 8.57739109e-08} & 7.74 \\
  & 2 & \num{ 8.38408929e-10} & 6.68 \\
  & 3 & \num{ 1.12690397e-11} & 6.22 \\
% \midrule
% 6 & 0 & \num{ 6.71285613e-06} &   - \\
%   & 1 & \num{ 6.32187890e-09} & 10.0 \\
% \midrule
\bottomrule
\end{tabular}
\end{center}
\end{minipage}
\hspace*{0.1\textwidth}
\caption{Initial mesh with approximated geometry configuration $\Omegalin$, $\Omegalin_{i}$ (top left) and isoparametrically mapped domain for $k=2$ (bottom left). The table shows $L^2$-norm discretization errors on  successively refined meshes and estimated orders of convergence.}
\label{fig:numex2d}
\end{figure}

\vspace*{-0.2cm}
\section*{Acknowledgement}
% \noindent
C. Lehrenfeld gratefully acknowledges funding by the German Science Foundation (DFG) within the project ``LE 3726/1-1''.
\vspace*{-0.2cm}
  
\bibliographystyle{siam}
\bibliography{literature}

\end{document}